%% file: sobolmatrix.tex
\documentclass{article}
\usepackage{amsmath,amsfonts,amssymb,amsthm}
\usepackage{paralist}
\usepackage{booktabs}
\usepackage[round]{natbib}
\usepackage{multirow}
\usepackage{kbordermatrix}

\input{newcommandsnotbook}

\newcommand{\ul}{\underline}
\newcommand{\ol}{\overline}
\newtheorem{theorem}{Theorem}
\newtheorem{proposition}{Proposition}

\theoremstyle{definition}

\newtheorem{definition}{Definition}

\newcommand{\sumdot}{\text{\tiny$\bullet$}}
\renewcommand{\glu}{{:}}

\newcommand{\myemph}[1]{\textsl{\textbf{#1}}}

\newcommand{\xor}{\mathrm{XOR}}
\newcommand{\nxor}{\mathrm{NXOR}}
\newcommand{\ult}{\underline\tau}
\newcommand{\olt}{\overline\tau}
\newcommand{\simiid}{\stackrel{\mathrm{iid}}{\sim}}
\newcommand{\bsone}{\boldsymbol{1}}
\newcommand{\bff}{\mathbf{F}}

\author{
Art B. Owen\\
Stanford University}
\date{April 2012}
\title{Variance components and generalized Sobol' indices}
\begin{document}

\maketitle
\begin{abstract}
This paper introduces generalized Sobol' indices,
compares strategies for their estimation, and
makes a systematic search for efficient estimators.
Of particular interest are contrasts, sums of
squares and indices of bilinear form
which allow a reduced number of function evaluations
compared to alternatives. The bilinear framework
includes some efficient estimators from
Saltelli (2002) and Mauntz (2002) as well as
some new estimators for specific variance
components and mean dimensions.
This paper also provides 
a bias corrected version of the estimator
of Janon et al.\,(2012) and extends the bias
correction to generalized Sobol' indices.
Some numerical comparisons are given.
\end{abstract}

\section{Introduction}

Sobol' indices are certain sums of variance
components in an ANOVA decomposition. They are used to
understand the importance of various subsets
of variables in global sensitivity analysis.
\cite{salt:ratt:andr:camp:cari:gate:sais:tara:2008}
give an extensive introduction to Sobol' indices
and variance based methods in general for 
investigating computer models.
Linear combinations of Sobol' indices
are also used to measure the effective
dimension of functions for quasi-Monte Carlo
integration. 

This article reviews Sobol' indices for a statistical
audience, relating them to well known ideas in
experimental design, particularly crossed
random effects. 
Moving from physical experiments to computer
experiments brings important changes in both
the costs and goals of the analysis. 
In physical experiments one may be interested
in all components of variance, or at least all
of the low order ones. In computer experiments
interest centers instead on sums of variance components.
While the ANOVA
for computer experiments is essentially the
same as that for physical ones, the
experimental designs are different.

Of particular interest are what are
known as `fixing methods' for estimation
of Sobol' indices.  These evaluate the
function at two points. Those two points
have identical random values for some of the 
input components (the ones that are `fixed') 
but have independently sampled values 
for the other components.
Sample variances and covariances of point pairs are
then used to estimate the Sobol' indices.

As a basic example, let $f$ be a deterministic function
on $[0,1]^5$. One kind of Sobol' index estimate takes a form
like
\begin{align}\label{eq:sobollowexamp}
\cov\bigl(
f(x_{1},x_{2},x_{3},x_{4},x_{5}),
f(x_{1},x_{2},x_{3},z_{4},z_{5})
\bigr)
\end{align}
for $x_{j}\simiid\dustd(0,1)$ independently
of $z_{j}\simiid\dustd(0,1)$.
As we will see below, 
this index measures the sum of variance components
over all subsets of the first three input variables.
The natural design to estimate~\eqref{eq:sobollowexamp}
consists of $n$ pairs of function evaluations,
which share the first $3$ input values, and have
independent draws in the last $2$ inputs.
In the language of statistical experimental
design \citep{bhh} this corresponds to $n$
independently generated $1\times 1\times 1\times 2^{2-1}$
designs. The first three variables are at $1$ level, while the
last two are a fractional factorial. Each replicate uses
different randomly chosen levels for the five variables.

An example of the second kind of Sobol' index is
\begin{align}\label{eq:sobolhighexamp}
&\frac12\var\bigl(
f(x_{1},x_{2},x_{3},x_{4},x_{5})-
f(x_{1},x_{2},x_{3},z_{4},z_{5})\bigr)\\
&= \var\bigl( f(x_{1},x_{2},x_{3},x_{4},x_{5})\bigr)
-\cov\bigl(f(x_{1},x_{2},x_{3},x_{4},x_{5}),
f(x_{1},x_{2},x_{3},z_{4},z_{5})\bigr)\notag
\end{align}
The sampling design to estimate~\eqref{eq:sobolhighexamp}
is that same as that for~\eqref{eq:sobollowexamp},
but the quantity estimated
is now the sum of all variance components
that involve any of the first three variables. 
The difference between~\eqref{eq:sobollowexamp}
and~\eqref{eq:sobolhighexamp} is that the
latter includes interactions between
the first three and last two variables
while the former excludes them.

The great convenience of Sobol's measures is
that they can be directly estimated by integration,
without explicitly estimating all of the
necessary interaction effects,
squaring them, integrating their squares
and summing those integrated squared estimates.
Sobol' provides a kind of tomography: integrals
of cross-products of $f$ reveal facts about the internal
structure of $f$.

The goal of this paper is to exhibit the entire
space of linear combinations of cross-moments
of function evaluations with some variables
fixed and others independently sampled.
Such a linear combination is a generalized
Sobol' index, or GSI.
Then, using this space of functions, we make
a systematic search for estimators of interpretable
quantities with desirable computational or
statistical properties.

This systematic approach yields some new
and useful estimators. Some have reduced
cost compared to previously known ones.
Some have reduced sampling variance.
It also encompasses some earlier work.
In particular, an efficient strategy to estimate
all two factor interaction mean squares
due to~\cite{salt:2002} appears as a special
case.

Section~\ref{sec:review} introduces some notation
and reviews the ANOVA of
$[0,1]^d$ and Sobol' indices. A compact notation
is necessary to avoid cumbersome expressions
with many indices.
Section~\ref{sec:gsi} defines the generalized
Sobol' indices and gives an expression for their
value. It also defines several special classes
of GSI based on interpretability, computational
efficiency, or statistical considerations.
These are contrasts, squares, sums of squares
and bilinear GSIs.
Section~\ref{sec:squares} shows that
many GSIs including the Sobol' index~\eqref{eq:sobollowexamp}
cannot be estimated by unbiased sums of squares. 
Section~\ref{sec:specific} considers estimation
of a specific variance component for a proper
subset containing $k$ of the  variables. A direct approach
requires $2^k$ function evaluations per
Monte Carlo sample, while a bilinear estimate
reduces the cost to $2^{\lfloor k/2\rfloor}+
2^{k-\lfloor k/2\rfloor}$. That section also
introduces a bilinear estimate for the superset
importance measure defined in Section~\ref{sec:review}.
Section~\ref{sec:od} considers some GSIs for
high dimensional problems. It includes a contrast
GSI which estimates the mean square dimension
of a function of $d$ variables using only $d+1$
function evaluations per Monte Carlo trial
as well as some estimators of the
mean dimension in the truncation sense.
Section~\ref{sec:biasc}
presents a bias correction for GSIs
that are not contrasts.
Section~\ref{sec:compare} makes some comparisons
among alternative methods and Section~\ref{sec:conc}
has conclusions.

\section{Background and notation}\label{sec:review}

The analysis of variance originates
with~\cite{fish:mack:1923}. It partitions
the variance of a quantity among all
non-empty subsets of factors, defined
on a finite Cartesian grid.

The ANOVA was generalized by~\cite{hoef:1948}
to functions in $L^2[0,1]^d$ for integer $d\ge1$.
That generalization extends the one for
factorial experimental designs in a natural way,
and can be applied to $L^2$ functions
on any tensor product domain.
For $d=\infty$, see~\cite{lss}.

The ANOVA of $L^2[0,1]^d$ is also attributed
to \cite{sobo:1969}.
For historical interest, we note
that Sobol' used a different approach than Hoeffding.
He represented $f$ by an expansion
in a complete orthonormal basis (tensor
products of Haar functions) and gathered
together terms corresponding to each
subset of variables.  That is, 
where Hoeffding has an analysis,
Sobol' has a synthesis.

We use $\bsx=(x_1,x_2,\dots,x_d)$
to represent a typical point in $[0,1]^d$.
The set of indices is $\cd=\{1,2,\dots,d\}$.
We write $u\subset v$ to denote
a proper subset, that is $u\subsetneq v$.
For $u\subseteq\cd$ we use $|u|$ to denote
the cardinality of $u$, and either $-u$ or
$u^c$ (depending on typographical clarity)
to represent the complementary set
$\cd-u$. The expression $u+v$ means
$u\cup v$ where $u$ and $v$ are understood to be disjoint.

For $u=\{j_1,j_2,\dots,j_{|u|}\}\subseteq\cd$ the point $\bsx_u\in[0,1]^{|u|}$
has components $(x_{j_1},x_{j_2},\dots,x_{j_{|u|}})$.
The differential $\mrd\bsx_u$ is $\prod_{j\in u}\mrd x_j$.

The ANOVA decomposition represents $f(\bsx)$ via
$$f(\bsx) = \sum_{u\subseteq\cd} f_u(\bsx)$$
where the functions $f_u$ are defined recursively by
\begin{align*}
f_u(\bsx) &= 
\int_{[0,1]^{d-|u|}}
\Bigl(f(\bsx)-\sum_{v\subset u}f_v(\bsx)\Bigr)\rd \bsx_{-u}\\
&=\int_{[0,1]^{d-|u|}}
f(\bsx)\rd \bsx_{-u}
-\sum_{v\subset u}f_v(\bsx).
\end{align*}
In statistical language, $u$ is a set of
factors and $f_u$ is the corresponding effect.
We get $f_u$ by subtracting sub-effects
of $f_u$ from $f$ and then averaging the
residual over $x_j$ for $j\not\in u$.

From usual conventions,
$f_\emptyset(\bsx) = \mu\equiv\int_{[0,1]^d}f(\bsx)\rd\bsx$
for all $\bsx\in [0,1]^d$.
The effect $f_u$ only depends on $x_j$ for those $j\in u$.
For $f\in L^2[0,1]^d$, these functions satisfy
$\int_0^1 f_u(\bsx)\rd x_j=0$, when $j\in u$
(proved by induction on $|u|$),
from which it follows that
$\int f_u(\bsx)f_v(\bsx)\rd\bsx=0$, for $u\ne v$.
The ANOVA identity is
$\sigma^2 = \sum_u\sigma^2_u$
where $\sigma^2=\int (f(\bsx)-\mu)^2\rd\bsx$,
$\sigma^2_\emptyset=0$ and
$\sigma^2_u = \int f_u(\bsx)^2\rd\bsx$
is the variance component for $u\ne\emptyset$.

We will need the following quantities.
\begin{definition}
For integer $d\ge1$,
let $u$ and $v$ be subsets of $\cd=\{1,\dots,d\}$.
Then the sets $\xor(u,v)$ and $\nxor(u,v)$ are
\begin{align*}
\xor(u,v) & = u\cup v -u\cap v\\
\nxor(u,v) & = (u\cap v)\bigcup\, (u^c\cap v^c).
\end{align*}
\end{definition}
These are the exclusive-or of $u$ and $v$
and its complement, respectively.
The $\nxor$ operation satisfies the following
easily verifiable properties:
$\nxor(u,v)=\nxor(v,u)$, and
if $w=\nxor(u,v)$ then $u=\nxor(v,w)$.

\subsection{Sobol' indices}

This section introduces the Sobol' indices
that we generalize, and mentions some of
the methods for their estimation.

There are various ways that one might
define the importance of a variable $x_j$.
The importance of variable $j\in\{1,\dots,d\}$
is due in part to $\sigma^2_{\{j\}}$, but also
due to $\sigma^2_u$ for other sets $u$
with $j\in u$.  More generally, we may
be interested in the importance of
$\bsx_u$ for a subset $u$ of the variables.

\cite{sobo:1993} 
introduced two measures of variable
subset importance, which we denote
\begin{align*}
\underline\tau_u^2 & = \sum_{v\subseteq u}\sigma_v^2,\quad\text{and}\quad
\overline\tau_u^2 = \sum_{v\cap u\ne\emptyset}\sigma_v^2.
\end{align*}
We call these the lower and upper Sobol' index
for the set $u$, respectively.
The lower index is a total of variance components
for $u$ and all of its subsets.
One interpretation is
$\ult^2_u = \var( \mu_u(\bsx))$
where $\mu_u(\bsx) = \e( f(\bsx)\mid \bsx_u)$.

The upper index counts every ANOVA component that
touches the set $u$ in any way.
If $\ult_u^2$ is large then the subset $u$
is clearly important. If $\olt_u^2$ is small, then
the subset $u$ is not important and 
\cite{sobo:tara:gate:kuch:maun:2007}
investigate the effects of fixing such $\bsx_u$ at
some specific values.
The second measure includes interactions between
$\bsx_u$ and $\bsx_{-u}$ while the first measure
does not.

These Sobol' indices satisfy $\ul\tau_u^2\le\ol\tau_u^2$
and $\ul\tau_u^2+\ol\tau_{-u}^2=\sigma^2$.
Sobol' usually normalizes these quantities,
yielding global sensitivity indices $\underline\tau_u^2/\sigma^2$
and $\overline\tau_u^2/\sigma^2$.  In this paper
we work mostly with unnormalized versions.

Sobol's original work was published
in~\cite{sobo:1990} before being translated
in~\cite{sobo:1993}.
\cite{ishi:homm:1990} independently considered
computation of $\ult^2_{\{j\}}$.

To estimate Sobol' indices, one pairs
the point $\bsx$ with a hybrid
point $\bsy$ that shares some but not
all of the components of $\bsx$.
We denote the hybrid point by $\bsy = \bsx_u\glu\bsz_{-u}$ where
$y_j = x_j$ for $j\in u$ and $y_u =z_j$
for $j\not\in u$.

From the ANOVA properties one can show directly that
\begin{align*}
\int_{[0,1]^{2d-|u|}} f(\bsx)f(\bsx_u\glu\bsz_{-u})\rd\bsx\rd\bsz_{-u}
& =
\sum_v\int_{[0,1]^{2d-|u|}} 
f_v(\bsx)f(\bsx_u\glu\bsz_{-u})\rd\bsx\rd\bsz_{-u}\\
&=\mu^2 +\ul\tau_u^2.
\end{align*}
This is also a special case of Theorem~\ref{thm:withnxor} below.
As a result
$
\ult^2_u = \cov( f(\bsx),f(\bsx_u\glu\bsz_{-u}))
$
and so one can estimate this Sobol' index by
\begin{align}\label{eq:janons}
\wh\ult^2_u =
\frac1n\sum_{i=1}^n f(\bsx_{i})f(\bsx_{i,u}\glu\bsz_{i,-u}) - \hat\mu^2,
\end{align}
for $\bsx_i,\bsz_i\simiid \dustd(0,1)^d$,
where $\hat\mu = (1/n)\sum_{i=1}^nf(\bsx_i)$.
It is even better to use
$\hat\mu = (1/n)\sum_{i=1}^n(f(\bsx_i)+f(\bsx_{i,u}\glu\bsz_{i,-u}))/2$
instead, as shown by~\cite{jano:klei:lagn:node:prie:2012:tr}.

Similarly one can show that
$$\ol\tau_u^2 =
\frac12\int_{[0,1]^{d+|u|}} (f(\bsx)-f(\bsx_{-u}\glu\bsz_{u}))^2\rd\bsx\rd\bsz_{u},
$$
and one gets the estimate
$$
\wh\olt_u^2 = \frac1{2n}\sum_{i=1}^n
(f(\bsx_i)-f(\bsx_{i,-u}\glu\bsz_{i,u}))^2,
$$
for $\bsx_i,\bsz_i\simiid\dustd(0,1)^d$.

The estimate $\wh\olt_u^2$ is unbiased
for $\olt_u^2$ but $\wh\ult_u^2$ above is not
unbiased for $\ult_u^2$. It has a bias equal
to $-\var(\hat\mu)$. If $\int |f(\bsx)|^4\rd\bsx<\infty$
then this bias is asymptotically negligible, but in cases
where $\ult^2_u$ is small, the bias may be important.

\cite{maun:2002} and 
\cite{kuch:feil:shah:maun:2011}
use an estimator for $\ult^2_u$ derived as a sample
version of the identity
\begin{align}\label{eq:mauntz}
\ult^2_u = \iint f(\bsx)(f(\bsx_u\glu\bsz_{-u})-f(\bsz))\rd\bsx\rd\bsz.
\end{align}
\cite{salt:2002} also mentions this estimator.
Here and below, integrals are by default over
$\bsx\in[0,1]^d$ and/or $\bsz\in[0,1]^d$ even
though some components $x_j$ or $z_j$ may not be required.
An estimator based on~\eqref{eq:mauntz} with
$\bsx_i,\bsz_i\simiid\dustd(0,1)^d$ is unbiased for
$\ult^2_u$, but  it requires
$3n$ evaluations of $f$ instead of the $2n$
required by~\eqref{eq:janons}
with either formula for $\hat\mu$.
Proposition~\ref{prop:janonunb} in Section~\ref{sec:biasc}
gives an unbiased variant on the
estimator of~\eqref{eq:janons}
using only $2$ function evaluations per
$(\bsx_i,\bsz_i)$ pair.

There are $2^d-1$ variance components $\sigma^2_u$
as well as $2^d-1$ Sobol' indices $\ult^2_u$
and $\olt^2_u$ of each type. 
We can recover any desired $\sigma^2_u$ as
a linear combination of $\ult^2_v$.
For example
$\sigma^2_{\{1,2,3\}} = 
\ult^2_{\{1,2,3\}}
-\ult^2_{\{1,2\}}
-\ult^2_{\{1,3\}}
-\ult^2_{\{2,3\}}
+\ult^2_{\{1\}}
+\ult^2_{\{2\}}
+\ult^2_{\{3\}}$.
More generally, we have the Moebius-type relation
\begin{align}\label{eq:moebius}
\sigma^2_u = \sum_{v\subseteq u}(-1)^{|u-v|}\ult^2_v.
\end{align}

Because $f$ is defined on a unit cube and can
be computed at any desired point, methods
other than simple Monte Carlo can be applied.
Quasi-Monte
Carlo (QMC) sampling (see~\cite{nied92})
can be used instead of plain Monte Carlo,
and \cite{sobo:2001} reports that QMC is more effective.
For functions $f$ that are very
expensive, a Bayesian numerical analysis approach
\citep{oakl:ohag:2004} based on a Gaussian process
model for $f$ is an attractive way to compute
Sobol' indices.

\subsection{Related indices}

Another measure of the importance of $\bsx_u$
is the superset importance measure
\begin{align}\label{eq:superset}
\Upsilon^2_u = \sum_{v\supseteq u}\sigma^2_v
\end{align}
used by~\cite{hook:2004}
to quantify the effect of dropping 
all interactions containing the set $u$
of variables from a black box function.

Sums of ANOVA components are also used in
quasi-Monte Carlo sampling.  QMC is, in general, 
more effective on integrands $f$ that are dominated
by their low order ANOVA components. Two
versions of $f$ that are equivalent in Monte
Carlo sampling may behave quite differently
in QMC.  For example the basis used to sample
Brownian paths has been seen to affect
the accuracy of QMC integrals
\citep{cafmowen,acbrogla97,imai:tan:2002}.

The function $f$ has effective dimension
$s$ in the superposition sense \citep{cafmowen}, if
$\sum_{|u|\le s}\sigma^2_u\ge (1-\epsilon)\sigma^2$.
Typically $\epsilon=0.01$ is used as a default.
Similarly, $f$ has effective dimension
$s$ in the truncation sense \citep{cafmowen}, if
$\sum_{u\subseteq \{1,2,\dots,s\}}\sigma^2_u
=\ult^2_{\{1,2,\dots,s\}}\ge (1-\epsilon)\sigma^2$.

It is much easier to estimate the mean dimension
(superposition sense)
defined as $\sum_{u} |u|\sigma^2_u/\sigma^2$
than the effective dimension,
because the mean dimension is a linear
combination of variance components.
The mean dimension also offers better resolution
than effective dimension. For instance, two functions
having identical effective dimension $2$ might have 
different mean dimensions, say $1.03$ and $1.05$.
Similarly
one can estimate a mean square dimension
$\sum_{u} |u|^2\sigma^2_u/\sigma^2$:

\begin{theorem}\label{thm:dimmoments}
\begin{align}
\sum_{j=1}^d\olt^2_{\{j\}} &= \sum_{u}|u|\sigma^2_u\label{eq:meandim}\\
\sum_{j=1}^d\sum_{k\ne j}\olt^2_{\{j,k\}}
&= 2(d-1)\sum_{u}|u|\sigma^2_u
-\sum_u|u|^2\sigma^2_u
\label{eq:vardim}
\end{align}
\end{theorem}
\begin{proof}
This follows from Theorem 2 of~\cite{meandim}.
\end{proof}

\section{Generalized Sobol' indices}\label{sec:gsi}

Here we consider a general family
of quadratic indices similar to
those of Sobol'.
The general form of these indices is
\begin{align}\label{eq:gensobo}
\sum_{u\subseteq\cd}\sum_{v\subseteq\cd}
\Omega_{uv}\iint f(\bsx_u\glu\bsz_{-u})f(\bsx_v\glu\bsz_{-v})
\rd\bsx\rd\bsz
\end{align}
for coefficients $\Omega_{uv}$.
If we think of $f(\bsx)$ as being the
standard evaluation, then the generalized
Sobol' indices~\eqref{eq:gensobo} are linear
combinations of all possible second
order moments of $f$ based on fixing
two subsets, $u$ and $v$, of input variables.

A matrix representation of~\eqref{eq:gensobo}
will be useful below.  First we introduce
the Sobol' matrix
$\Theta\in\real^{2^d\times 2^d}$
with entries
$\Theta_{uv} = 
\iint f(\bsx_u\glu\bsz_{-u})f(\bsx_v\glu\bsz_{-v})
\rd\bsx\rd\bsz$ indexed by subsets $u$ and $v$.
Then~\eqref{eq:gensobo} is the matrix inner product
$\tr(\Omega^\tran\Theta)$ for the matrix
$\Omega$.
Here and below, we use matrices and vectors
indexed by the $2^d$ subsets of $\cd$.
The order in which subsets appear 
is not specified; any consistent ordering is
acceptable.

The Sobol' matrix is symmetric.  It also satisfies
$\Theta_{uv} = \Theta_{u^cv^c}$.
Theorem~\ref{thm:withnxor} gives the general form
of a Sobol' matrix entry.

\begin{theorem}\label{thm:withnxor}
Let $f\in L^2[0,1]^d$ for $d\ge 1$, 
with mean $\mu=\int f(\bsx)\rd\bsx$
and variance components $\sigma^2_u$
for $u\subseteq\cd$.
Let $u,v\subseteq\cd$. Then
the $uv$ entry of the Sobol' matrix is
$$\Theta_{uv} 
= \mu^2+\underline\tau^2_{\nxor(u,v)}.$$
\end{theorem}
\begin{proof}
First,
\begin{align*}
\Theta_{uv} &=
\iint f(\bsx_u\glu\bsz_{-u})f(\bsx_v\glu\bsz_{-v})\rd\bsx\rd\bsz\\
& =\sum_{w\subseteq\cd}\sum_{w'\subseteq\cd}
\iint f_w(\bsx_u\glu\bsz_{-u})
f_{w'}(\bsx_v\glu\bsz_{-v})\rd\bsx\rd\bsz\notag\\
& =\sum_{w\subseteq\cd}
\iint f_w(\bsx_u\glu\bsz_{-u})f_{w}(\bsx_v\glu\bsz_{-v})\rd\bsx\rd\bsz,
\end{align*}
because if $w\ne w'$, then there is an index
$j\in\xor(w,w')$, 
for which either the integral over $x_j$
or the integral over $z_j$ of $f_wf_{w'}$ above vanishes.

Next, suppose that $w$ is not a subset
of $\nxor(u,v)$.
Then there is an index $j\in\xor(u,v)\cap w$.
If $j\in w\cap u\cap v^c$, then the integral over $x_j$
vanishes, while if $j\in w\cap v\cap u^c$,
then the integral over $z_j$ vanishes.
Therefore
\begin{align*}
\Theta_{uv}
&=\sum_{w\subseteq\nxor(u,v)}
\iint f_w(\bsx_u\glu\bsz_{-u})f_{w}(\bsz_v\glu\bsx_{-v})\rd\bsx\rd\bsz\\
&=\sum_{w\subseteq\nxor(u,v)}\int f_w(\bsx)^2\rd\bsx\\
&=\mu^2+\underline\tau^2_{\nxor(u,v)}.\qedhere
\end{align*}
\end{proof}

\subsection{Special GSIs}

Equation~\eqref{eq:gensobo} describes a
$2^{2d}$ dimensional family of linear
combinations of pairwise function products.
There are only $2^d-1$ ANOVA components to estimate.
Accordingly we are interested in special
cases of~\eqref{eq:gensobo} with desirable properties.

A GSI is a \myemph{contrast} if
$\sum_u\sum_v\Omega_{uv}=0$.  
Contrasts are unaffected
by the value of the mean $\mu$,
and so they lead to unbiased estimators
of linear combinations of variance components.
GSIs that are not contrasts contain
a term $\mu^2\sum_u\sum_v\Omega_{uv}$
and require us to subtract an estimate
 $\hat\mu^2\sum_u\sum_v\Omega_{uv}$
in order to estimate 
$\sum_u\sum_v\Omega_{uv}\ult^2_{\nxor(u,v)}$
as in Section~\ref{sec:biasc}.

A generalized Sobol' index is a \myemph{square} if
it takes the form
$$
\iint \biggl(\,\sum_u\lambda_uf(\bsx_u\glu\bsz_{-u})\biggr)^2\rd\bsx\rd\bsz.
$$
Squares and sums of squares have the advantage
that they are non-negative and hence avoid the
problems associated with negative sample variance
components. 
A square GSI can be written compactly as
$\tr( \lambda\lambda^\tran\Theta)=\lambda^\tran\Theta\lambda$
where $\lambda$ is a vector of $2^d$ coefficients.
If $\lambda_u$ is sparse (mostly zeros)
then a square index is inexpensive to compute.

A generalized Sobol' index is \myemph{bilinear} if
it takes the form
$$
\iint 
\biggl(\,\sum_u\lambda_uf(\bsx_u\glu\bsz_{-u})\biggr)
\biggl(\,\sum_v\gamma_vf(\bsx_v\glu\bsz_{-v})\biggr)
\rd\bsx\rd\bsz.
$$
Bilinear estimates have the advantage
of being rapidly computable. If there
are $\Vert \lambda\Vert_0$ nonzero elements
in $\lambda$
and  $\Vert \gamma\Vert_0$ nonzero elements
in $\gamma$ then the integrand in 
a bilinear generalized Sobol' index can
be computed with at most 
$\Vert\gamma\Vert_0+\Vert\lambda\Vert_0$ function calls
and sometimes fewer (see Section~\ref{sec:costs})
even though it combines values from
$\Vert\gamma\Vert_0\times\Vert\lambda\Vert_0$ function pairs.
We can write the bilinear GSI as
$\tr( \lambda\gamma^\tran\Theta) =\gamma^\tran\Theta\lambda$.
The sum of a small number of bilinear GSIs is
a low rank GSI.

A GSI is \myemph{simple} if it is written
as a linear combination of entries in just
one row or just one column of $\Theta$,
such as
$$
\iint 
\sum_u\lambda_uf(\bsx_u\glu\bsz_{-u})f(\bsz)
\rd\bsx\rd\bsz.
$$
It is convenient if the
chosen row or column corresponds to $u$
or $v$ equal to $\emptyset$ or $\cd$.
Any linear combination 
$\sum_u\delta_u(\mu^2+\ult^2_u)$
of variance components and $\mu^2$ can be written as a simple GSI
taking $\lambda_u = \delta_{-u}$.
There are computational advantages to some non-simple 
representations.

\subsection{Sample GSIs}

To estimate a GSI we take
pairs $(\bsx_i,\bsz_i)\simiid\dustd(0,1)^{2d}$
for $i=1,\dots,n$
and compute
$\tr(\Omega^\tran\wh\Theta)$
where
$$
\wh\Theta_{uv} = \frac1n\sum_{i=1}^n f(\bsx_{i,u}\glu\bsz_{i,-u})f(\bsx_{i,v}\glu\bsz_{i,-v}).
$$
We can derive a matrix expression for the
estimator by introducing the vectors
$$F_i \equiv F(\bsx_i,\bsz_i) = 
\bigl(f(\bsx_{i,u}\glu\bsz_{i,-u})\bigr)_{u\subseteq\cd}\in\real^{2^d\times 1},$$
for $i=1,\dots,n$
and the matrix
$$\bff =\begin{pmatrix}F_1 & F_2 & \cdots & F_n\end{pmatrix}^\tran
\in\real^{n\times 2^d}.$$
The vectors $F_i$ have covariance $\Theta-\mu^2$. Then
$$\wh\Theta = \frac1n\sum_{i=1}^nF_iF_i^\tran = \frac1n\bff^\tran\bff,$$
and the sample GSI is
$$\tr(\Omega^\tran\wh\Theta)
=\tr(\wh\Theta^\tran\Omega)
=\frac1n\tr(\bff^\tran\bff\Omega)
=\frac1n\tr(\bff^\tran\Omega\bff).$$

\subsection{Cost per $(\bsx,\bsz)$ pair}\label{sec:costs}

We suppose that the cost of computing a
sample GSI is dominated by the number
of function evaluations required.
If the GSI requires $C(\Omega)$ (defined below)
distinct function evaluations
per pair $(\bsx_i,\bsz_i)$ for $i=1,\dots,n$,
then the cost of the sample GSI is proportional to $nC(\Omega)$.

If the row $\Omega_{uv}$ for given $u$
and all values of $v$ is not entirely zero
then we need the value $f(\bsx_u\glu\bsz_{-u})$.
Let 
$$C_{u\sumdot}(\Omega)
=\begin{cases} 1, & \exists v\subseteq\cd\quad\text{with}\quad 
\Omega_{uv}\ne 0\\
0,&\text{else}
\end{cases}
$$
indicate whether $f(\bsx_u\glu\bsz_{-u})$ is needed
as the `left side' of a product
$f(\bsx_u\glu\bsz_{-u})f(\bsx_v\glu\bsz_{-v})$.
The number of function evaluations required
for the GSI $\tr(\Omega^\tran\Theta)$ is:
\begin{align}\label{eq:costdef}
C(\Omega) = 
\sum_{u\subseteq\cd} 
\Bigl(
C_{u\sumdot}(\Omega)+C_{u\sumdot}(\Omega^\tran)
-C_{u\sumdot}(\Omega)C_{u\sumdot}(\Omega^\tran)
\Bigr).
\end{align}
We count the number of rows of $\Omega$
for which $f(\bsx_u\glu\bsz_{-u})$
is needed, add the number of columns
and then subtract the number of double counted
sets $u$.

\section{Squares and sums of squares}\label{sec:squares}

A square or sum of squares yields a nonnegative
estimate.
An unbiased and nonnegative estimate is especially
valuable. 
When the true GSI is zero, 
an unbiased nonnegative estimate will always
return exactly zero as~\cite{frut:rous:kuhn:2012:tr}
remark.
The Sobol' index $\olt^2_u$ is of square
form, but $\ult^2_u$ is not.
Theorem~\ref{thm:dimmoments}
leads to a sum of squares for $\sum_u|u|\sigma^2_u$.

\cite{meandim} express the superset
importance measure as a square:
\begin{align}\label{eq:upassquare}
\Upsilon^2_u = 
\frac1{2^{|u|}}
\iint 
\biggl|
\sum_{ v\subseteq u}
(-1)^{|u-v|}f(\bsx_v\glu\bsz_{-v})\biggr|^2
\rd\bsx\rd\bsz.
\end{align}
\cite{frut:rous:kuhn:2012:tr} find that
a sample version of~\eqref{eq:upassquare}
is the best among four estimators of $\Upsilon^2_u$.

In classical crossed mixed effects models 
\citep{mont:1997}
every ANOVA expected mean square
has a contribution from the highest order
variance component, typically containing measurement error.
A similar phenomenon applies for Sobol' indices.
In particular, no square GSI or sum of squares will yield $\ult^2_u$
for $|u|<d$, as the next proposition shows.

\begin{proposition}
The coefficient of $\sigma^2_{\cd}$ in
$\sum_{r=1}^R \iint (\sum_u\lambda_{r,u}f(\bsx_u\glu\bsz_{-u}))^2\rd\bsx\rd\bsz$
is
$\sum_{r=1}^R\sum_u\lambda_{r,u}^2$.
\end{proposition}
\begin{proof}
It is enough to show this for $R=1$ with
$\lambda_{1,u}=\lambda_u$.
First, 
\begin{align}\label{eq:onesquare}
\iint (\sum_u\lambda_{u}f(\bsx_u\glu\bsz_{-u}))^2\rd\bsx\rd\bsz
& = \lambda^\tran\Theta\lambda.
\end{align}
Next, the only elements of $\Theta$ containing
$\sigma^2_{\cd}$ are the diagonal ones, equal
to $\sigma^2$.
Therefore the coefficient of $\sigma_{\cd}^2$
in~\eqref{eq:onesquare} is $\sum_u\lambda^2_u$.
\end{proof}

For a square or sum of
squares to be free of $\sigma^2_{\cd}$
it is necessary to have $\sum_{r=1}^R\sum_u\lambda_{r,u}^2=0$.
That in turn requires all the $\lambda_{r,u}$ to vanish,
leading to the degenerate case $\Omega=0$.
As a result, we cannot get an unbiased
sum of squares for any GSI that does not
include $\sigma^2_{\cd}$. In particular,
$\ult^2_u$ cannot have an unbiased
sum of squares estimate for $|u|<d$.


\section{Specific variance components}\label{sec:specific}

For any $w\subseteq\cd$ the variance component
for $w$ is given in~\eqref{eq:moebius}
as an alternating sum of $2^{|w|}$
lower Sobol' indices.
It can thus be estimated by a simple GSI,
\begin{align}\label{eq:simplevarcomp}
f(\bsx)\sum_{v\subseteq w}\lambda_vf(\bsx_{v}\glu\bsz_{-v})
\end{align}
where $\lambda_v = (-1)^{|w-v|}$.
The cost of this simple GSI is $C=2^{|w|}+1_{|w|<d}$.
If $w=\cd$, then $f(\bsx)$ appears twice, but
otherwise it is only used once.
The GSI can also be estimated by some bilinear GSIs
using fewer function evaluations as we show here.

We begin by noting that for $u,v\subseteq w$,
\begin{align*}
\nxor(u,v+w^c) = (\xor(u,v)+w^c)^c = \nxor(u,v)\cap w
\end{align*}
and so 
$$\Theta_{u,v+w^c} = \mu^2+\ult^2_{\nxor(u,v)\cap w}$$
does not involve any of the variables $x_j$ for $j\not\in w$.
Note especially that $\nxor(u,v)$ itself contains all of $w^c$
and is {\sl not} helpful in estimating $\sigma^2_w$ when 
$|w|<d$.

To illustrate, suppose that $w=\{1,2,3\}$.
Let $u,v\subseteq w$. Then we can
work out a $2\times 8$
submatrix of the Sobol' matrix using
\begin{align}\label{eq:2x8sub}
\kbordermatrix{\nxor(u,v+w^c) &\emptyset&1&2&3&12&13&23&123\\
\emptyset & 123 & 23  & 13 & 12 & 3 & 2 & 1 & \emptyset\\
1         &  23 & 123 & 3  & 2  & 13& 12 & \emptyset & 1\\
}
\end{align}
where we omit braces and commas from the set notation.

It follows now that we can use a non-simple bilinear GSI,
$\lambda^\tran\Theta\gamma$, where 
$\lambda$ and $\gamma$ are given by
\begin{align}\label{eq:alttriple}
\kbordermatrix{ &\emptyset&\phm1&\phm2&\phm3&\phm12&\phm13&\phm23&\phm123\\
\lambda_u  & 1 & -1     & \phm0  & \phm0  & \phm0& \phm0 &\phm0 & \phm0\\
\gamma_v   & 1 & \phm0  & -1 & -1 & \phm0 & \phm0 & \phm1 & \phm0\\
}
\end{align}
with $u$ and $v-w^c$ given along the top labels
in~\eqref{eq:alttriple} while the 
remaining $2^d-8$ elements of $\lambda$
and $\gamma$ are all zero.
Specifically, the expected value of
\begin{align}\label{eq:bilin1}
\sum_{u\subseteq\{1\}}\sum_{v\subseteq\{2,3\}}
(-1)^{|u|+|v|}
f(\bsx_u\glu\bsz_{-u})f(\bsx_{v+w^c}\glu\bsz_{v^c-w})
\end{align}
is $\sigma^2_{\{1,2,3\}}$.
While equation~\eqref{eq:simplevarcomp} for $w=\{1,2,3\}$
requires $9$ function
evaluations per $(\bsx,\bsz)$ pair, equation~\eqref{eq:bilin1}
only requires $6$ function evaluations. For $|w|<d$, the $u=\emptyset$
and $v=\emptyset$ evaluations are different due to
the presence of $w^c$, so no
evaluations are common to both the $\lambda$ and $\gamma$
expressions.
There are also two variants of~\eqref{eq:alttriple}
that single out variables
$2$ and $3$ respectively, analogously to the
way that~\eqref{eq:bilin1} treats variable~$1$.

In general, bilinear GSIs let us
estimate $\sigma^2_w$ using $2^k+2^{|w|-k}$
function evaluations per $(\bsx,\bsz)$ pair
for integer $1\le k<|w|$ instead of
the $2^{|w|}$ evaluations that a simple GSI
requires.

\begin{theorem}\label{thm:bilinvarcomp}
Let $w$ be a nonempty subset of $\cd$ for $d\ge1$.
Let $f\in L^2[0,1]^d$. Choose $w_1\subseteq w$
and put $w_2=w-w_1$.
Then
\begin{align}\label{eq:bilshortcut}
\sigma^2_w =
\sum_{u_1\subseteq w_1}\sum_{u_2\subseteq w_2}
(-1)^{|u_1|+|u_2|}
\iint f(\bsx_{u_1}\glu\bsz_{-u_1})
f(\bsx_{u_2+w^c}\glu\bsz_{u_2^c-w})\rd\bsx\rd\bsz.
\end{align}
\end{theorem}
\begin{proof}
Because the right hand side of~\eqref{eq:bilshortcut} is
a contrast, we can assume that $\mu=0$. Then
$\nxor(u_1,u_2+w^c) =\nxor(u_1,u_2)\cap w
=w-(u_1+u_2)$. Therefore
\begin{align*}
\phantom{=}\sum_{u_1\subseteq w_1}\sum_{u_2\subseteq w_2}
(-1)^{|u_1|+|u_2|}\Theta_{u_1,u_2+w^c}\notag
&=\sum_{u_1\subseteq w_1}\sum_{u_2\subseteq w_2}
(-1)^{|u_1|+|u_2|}\,\ult^2_{w-(u_1+u_2)}\\
&=\sum_{u_1\subseteq w_1}\sum_{u_2\subseteq w_2}
(-1)^{|w_1-u_1|+|w_2-u_2|}\,\ult^2_{u_1+u_2}
\end{align*}
after a change of variable from $u_j$ to $w_j-u_j$ for $j=1,2$.
We may write the above as
\begin{align}\label{eq:shortcutexpr}
\sum_{u_1\subseteq w_1}\sum_{u_2\subseteq w_2}
(-1)^{|w_1-u_1|+|w_2-u_2|}\,\sum_{v\subseteq u_1+u_2}\sigma^2_v.
\end{align}

Consider the set $v\subseteq\cd$. The coefficient
of $\sigma^2_v$ in~\eqref{eq:shortcutexpr} is $0$
if $v\cap w^c\ne\emptyset$. Otherwise, we may write
$v=v_1+v_2$ where $v_j\subseteq w_j$, $j=1,2$.
Then the coefficient of $\sigma^2_v$ in~\eqref{eq:shortcutexpr} is
\begin{align*}
&\sum_{u_1\subseteq w_1}\sum_{u_2\subseteq w_2}
(-1)^{|w_1-u_1|+|w_2-u_2|}1_{v\subseteq u_1+u_2}\\
&=
\sum_{u_1: v_1\subseteq u_1\subseteq w_1}(-1)^{|w_1-u_1|}
\sum_{u_2: v_2\subseteq u_2\subseteq w_2}(-1)^{|w_2-u_2|}.
\end{align*}
These alternating sums over $u_j$ with
$v_j\subseteq u_j\subseteq w_j$ equal
$1$ if $v_j=w_j$ but otherwise they are zero.
Therefore the coefficient of $\sigma^2_v$ in~\eqref{eq:shortcutexpr}
is $1$ if $v=w$ and is $0$ otherwise.
\end{proof}

We can use
Theorem~\ref{thm:bilinvarcomp} to get a bilinear
(but not square)
estimator of $\sigma^2_{\cd} = \Upsilon^2_{\cd}$.
A similar argument to that in Theorem~\ref{thm:bilinvarcomp}
yields a bilinear estimator of superset importance
$\Upsilon^2_w$ for a general set $w$.

\begin{theorem}\label{thm:bilinsuper}
Let $w$ be a nonempty subset of $\cd$ for $d\ge1$.
Let $f\in L^2[0,1]^d$. Choose $w_1\subseteq w$
and put $w_2=w-w_1$.
Then
\begin{align}\label{eq:bilshortcut2}
\Upsilon^2_w =
\sum_{u_1\subseteq w_1}\sum_{u_2\subseteq w_2}
(-1)^{|u_1|+|u_2|}
\Theta_{w^c+u_1,w^c+u_2}.
\end{align}
\end{theorem}
\begin{proof}
Because $w\ne\emptyset$, the estimate is a 
contrast and so we may suppose $\mu=0$.
For disjoint $u_1,u_2\subseteq w$,
$\nxor(w^c+u_1,w^c+u_2)
= \cd-u_1-u_2$,
and so the right side of~\eqref{eq:bilshortcut2} equals
\begin{align*}
\sum_{u_1\subseteq w_1}\sum_{u_2\subseteq w_2}
(-1)^{|u_1|+|u_2|}\ult^2_{\cd-u_1-u_2}
&=\sum_{u_1\subseteq w_1}\sum_{u_2\subseteq w_2}
(-1)^{|u_1|+|u_2|}\sum_{v\subseteq\cd-u_1-u_2}\sigma^2_v\\
&=\sum_{v}\sigma^2_v
\sum_{u_1\subseteq w_1}\sum_{u_2\subseteq w_2}
(-1)^{|u_1|+|u_2|}1_{v\subseteq\cd-u_1-u_2}.
\end{align*}
Now write $v=(v\cap w^c)+v_1+v_2$ with $v_j\subseteq w_j$, $j=1,2$.
The coefficient of $\sigma^2_v$ is
\begin{align*}
\sum_{u_1\subseteq w_1}\sum_{u_2\subseteq w_2}
(-1)^{|u_1|+|u_2|}1_{u_1\cap v_1=\emptyset}1_{u_2\cap v_2=\emptyset}.
\end{align*}
Now
\begin{align*}
\sum_{u_1\subseteq w_1}(-1)^{|u_1|}1_{u_1\cap v_1=\emptyset}
&=
\sum_{u_1\subseteq w_1-v_1}(-1)^{|u_1|}
\end{align*}
which vanishes unless $w_1=v_1$ and otherwise
equals $1$. Therefore the coefficient
of $\sigma^2_v$ is $1$ if $v\supseteq w$
and is $0$ otherwise.
\end{proof}

The cost of the estimator~\eqref{eq:bilshortcut2}
is $C = 2^{|w_1|}+2^{|w_2|}-1$, because the
evaluation  $f(\bsx_{w^c}\glu\bsz_w)$ can be
used for both $u_1=\emptyset$ and $u_2=\emptyset$.

\section{GSIs with $O(d)$ function evaluations per pair}\label{sec:od}

Some problems, like computing mean dimension,
can be solved with $O(d)$ different integrals
instead of the $O(2^d)$ required to estimate
all ANOVA components.
In this section we enumerate what can be
estimated by certain GSIs based on only
$O(d)$ carefully chosen 
function evaluations per $(\bsx_i,\bsz_i)$
pair.

\subsection{Cardinality restricted GSIs}
One way to reduce function evaluations
to $O(d)$ is to consider only subsets
$u$ and $v$ with cardinality $0$,
$1$, $d-1$, or $d$.
We suppose for $d\ge2$ that $j$ and $k$
are distinct elements of~$\cd$.
Letting $j$ and  $k$ substitute
for $\{j\}$ and $\{k\}$ respectively
we can enumerate $\nxor(u,v)$ for
all of these subsets as follows:
\begin{align*}
\kbordermatrix{\nxor &\emptyset&j&k&-j&-k&\cd\\
\emptyset & \cd & -j & -k & j & k &\emptyset\\
j   &  -j & \cd & -\{j,k\} & \emptyset & \phm\{j,k\} & j\\
-j   &  j & \emptyset & \phm\{j,k\} & \cd & -\{j,k\} &-j\\
\cd & \emptyset & j & k& -j & -k & \cd},
\end{align*}
and hence the accessible elements of the Sobol' matrix are:
\newcommand{\someex}{0.5ex}
\begin{align}\label{eq:availcard}
\kbordermatrix{\Theta_{uv}-\mu^2 &\emptyset&j&k&-j&-k&\cd\\[\someex]
\emptyset & \sigma^2 & \ult^2_{-j} & \ult^2_{-k} & \ult^2_j & \ult^2_k &0\\[\someex]
j   &  \ult^2_{-j} & \sigma^2 & \ult^2_{-\{j,k\}} & 0 & \ult^2_{\{j,k\}} & \ult^2_j\\[\someex]
-j   &  \ult^2_j & 0 & \ult^2_{\{j,k\}} & \sigma^2 & \ult^2_{-\{j,k\}} &\ult^2_{
-j}\\[\someex]
\cd & 0 & \ult^2_j & \ult^2_k& \ult^2_{-j} & \ult^2_{-k} & \sigma^2}.
\end{align}

Using~\eqref{eq:availcard} we can construct
estimates of $\sum_u|u|\sigma^2_u = \sum_j\olt^2_{\{j\}}$
and $\sum_{|u|=1}\sigma^2_u=\sum_j\ult^2_{\{j\}}$ at cost $C=d+1$.
Simple GSI estimates
are available using $u = \emptyset$
and either $v=\{j\}$ or $v=-\{j\}$ for $j=1,\dots,d$.
More interestingly, it is possible to compute
all $d(d-1)/2$ indices $\olt^2_{\{j,k\}}$
along with all $\ult^2_{\{j\}}$
and $\olt^2_{\{j\}}$ for $j=1,\dots,d$, at
total cost $C=d+2$ as was first shown
by~\citet[Theorem 1]{salt:2002}.
Given $C=2d+2$ evaluations one can also compute
all of the $\ult^2_{\{j,k\}}$ indices by 
pairing up $u=\{j\}$ and $v=-\{k\}$
\cite[Theorem 2]{salt:2002}.

For the remainder of this section we present
some contrast estimators.
The estimate
$$
\frac1{2n}\sum_{j=1}^d 
\bigl( f(\bsx)-f(\bsx_{\{-j\}}\glu\bsz_{\{j\}})\bigr)^2.
$$
is both a contrast and a sum of squares.  It has
expected value  $\sum_u|u|\sigma^2_u$
and cost $C=d+1$.

Next, to estimate
$\sum_{u}1_{|u|=1}\sigma^2_u$ 
by a contrast using
$d+2$ function evaluations per $(\bsx_i,\bsz_i)$
pair, let
$$\lambda_u = 
\begin{cases}
1, & |u|=1\\
-d, & |u|=d.
\end{cases}
$$
Then the contrast $\sum_u\lambda_u f(\bsx_u\glu\bsz_{-u})f(\bsz)$
has expected value 
$$\sum_{j=1}^d\ult^2_{\{j\}}
=\sum_{j=1}^d\sigma^2_{\{j\}}
=\sum_{|u|=1}\sigma^2_u.$$

The total of second order interactions
can be estimated with a contrast at cost $C=2d+2$.
Taking 
\begin{align*}
\lambda_u = 
\begin{cases}
1, & |u| = 1\\
-d, & |u| = 0\\
0,  &\text{else,}
\end{cases}\quad\text{and}\quad
\gamma_v = 
\begin{cases}
1, & |v| = d-1\\
-(d-2), & |v| = d\\
0, & \text{else}
\end{cases}
\end{align*}
we get a contrast with
\begin{align*}
\lambda^\tran\Theta\gamma & =
\sum_{j=1}^d\sum_{k=1}^d \ult^2_{\{j,k\}}1_{j\ne k}
-d\sum_{k=1}^d\ult^2_{\{k\}}
-(d-2)\sum_{j=1}^d\ult^2_{\{j\}}\\
& = 2\sum_{|u|=2}\ult^2_u-(2d-2)\sum_{|u|=1}\ult^2_u
 = 2\sum_{|u|=2}\sigma^2_u.
\end{align*}
Thus $\lambda^\tran\wh\Theta\gamma/2$ estimates
$\sum_{|u|=2}\sigma^2_u$, at cost $C=2d+2$.

Next, taking 
\begin{align*}
\lambda_u = 
\begin{cases}
1, & |u| = 1\\
-d, & |u| = 0\\
0,  &\text{else,}
\end{cases}\quad\text{and}\quad
\gamma_v = 
\begin{cases}
1, & |v| = 1\\
-(d-2), & |v| = 0\\
0, & \text{else}
\end{cases}
\end{align*}
we get a contrast with
\begin{align*}
\lambda^\tran\Theta\gamma & =
\sum_{j=1}^d\sum_{k=1}^d \ult^2_{-\{j,k\}}1_{j\ne k}
+\sum_{j=1}^d\sigma^2
-d\sum_{k=1}^d\ult^2_{-\{k\}}
-(d-2)\sum_{j=1}^d\ult^2_{-\{j\}}+d(d-2)\sigma^2\\
& =
d^2\sigma^2-\sum_{j=1}^d\sum_{k=1}^d \olt^2_{\{j,k\}}1_{j\ne k} 
-2d(d-1)\sigma^2+2(d-1)\sum_{j=1}^d\olt^2_{\{j\}}+d(d-2)\sigma^2\\
& =2(d-1)\sum_{j=1}^d\olt^2_{\{j\}}
-\sum_{j=1}^d\sum_{k=1}^d \olt^2_{\{j,k\}}1_{j\ne k} 
=2\sum_{u}|u|^2\sigma^2_{u},
\end{align*}
using Theorem~\ref{thm:dimmoments}.
Therefore $\e\bigl( \lambda^\tran\wh\Theta\gamma/2\bigr)
= \sum_{u}|u|^2\sigma^2_u$, at cost $C=d+1$.

\subsection{Consecutive index GSIs}
A second way to reduce function evaluations
to $O(d)$ is to consider only subsets
$u$ and $v$ of the form $\{1,2,\dots,j\}$
and $\{j+1,\dots,d\}$.
We write these as $(0,j]$ and $(j,d]$ respectively.
If $f(\bsx)$ is the result of a process evolving
in discrete time then $(0,j]$ represents the
effects of inputs up to time $j$ and
$(j,d]$ represents those after time $j$.
A small value of $\olt^2_{(0,j]}$ then
means that the first $j$ inputs are nearly forgotten
while a large value for $\ult^2_{(0,j]}$ means
the initial conditions have a lasting effect.

We suppose for $d\ge2$ that $1\le j<k\le d$.
Once again, we can enumerate $\nxor(u,v)$ for
all of the subsets of interest:
\begin{align}\label{eq:nxorconsec}
\kbordermatrix{\nxor &\emptyset&(0,j]&(0,k]&(j,d]&(k,d]&\cd\\
\emptyset & \cd & (j,d] & \phm(k,d] & (0,j] & \phm(0,k] &\emptyset\\
(0,j]   &  (j,d] & \cd & -(j,k] & \emptyset & \phm(j,k] & (0,j]\\
(j,d]   &  (0,j] & \emptyset & \phm(j,k] & \cd & -(j,k] & (j,d]\\
\cd & \emptyset & (0,j] & \phm(0,k]& (j,d] & \phm(k,d] & \cd},
\end{align}
and hence the accessible elements of the Sobol' matrix are
$\mu^2+\ult^2_u$ for the sets $u$ in the array above.

The same strategies used on singletons and their
complements can be applied to 
consecutive indices. They yield interesting
quantities related to mean dimension in
the truncation sense. To describe them,
we write 
$\lfloor u\rfloor = \min\{ j \mid j\in u\}$
and
$\lceil u\rceil = \max\{ j \mid j\in u\}$
for the least and greatest indices in the
non-empty set $u$.

\begin{proposition}\label{prop:meanupdn}
Let $f\in L^2[0,1]^d$ have variance components $\sigma^2_u$.
Then
\begin{align*}
\sum_{j=1}^{d-1}\bigl(\Theta_{(0,j],\cd} -\Theta_{\emptyset,\cd}\bigr)
& = \sum_{u\subseteq \cd}(d-\lceil u\rceil)\sigma^2_u,\quad\text{and,}\\
\sum_{j=1}^{d-1}\bigl(\Theta_{(j,d],\cd} -\Theta_{\emptyset,\cd}\bigr)
&= \sum_{u\subseteq \cd}(\lfloor u\rfloor-1)\sigma^2_u.
\end{align*}
\end{proposition}
\begin{proof}
Since these are contrasts, we may suppose that $\mu=0$.
Then, using~\eqref{eq:nxorconsec}
\begin{align*}
\sum_{j=1}^{d-1}\Theta_{(0,j],\cd}
& 
= \sum_{j=1}^{d-1}\ult^2_{(0,j]}
= (d-1)\sigma^2-\sum_{j=1}^{d-1}\olt^2_{(j,d]}.
\end{align*}
Next, $\Theta_{\emptyset,\cd}=0$, and
\begin{align*}
\sum_{j=1}^{d-1} \olt^2_{(j,d]}
= \sum_{u\subseteq\cd}\sigma^2_u\sum_{j=1}^{d-1}1_{\{j+1,\dots,d\}\cap u\ne\emptyset}
= \sum_{u\subseteq \cd}(\lceil u\rceil-1)\sigma^2_u.
\end{align*}
Combining these yields the first result. The second
is similar.
\end{proof}

Using Proposition~\ref{prop:meanupdn} we can
obtain an estimate of 
$\sum_u\lceil u\rceil\sigma^2_u/\sigma^2$,
the mean dimension  of $f$ in the
truncation sense. We also obtain a contrast
$$
\sum_{j=1}^{d-1}(\Theta_{\cd,\cd}-\Theta_{(0,j],\cd}-\Theta_{(j,d],\cd}
+2\Theta_{\emptyset,\cd})
= \sum_{u}(\lceil u\rceil-\lfloor u\rfloor)\sigma^2_u
$$
which measures the extent to which
indices at distant time lags contribute important interactions.

We can also construct GSIs based on pairs of segments.
For example,
\begin{align*}
\sum_{j=0}^{d-1}\sum_{k=j+1}^d\Theta_{(0,j],(k,d]}
-\frac{d(d-1)}2\mu^2
&=
\sum_u\sigma^2_u\sum_{j=0}^{d-1}\sum_{k=j+1}^d
1_{u\subseteq (j,k]}\\
&=
\sum_u\sigma^2_u\lfloor u\rfloor\bigl(d-\lceil u\rceil+1\bigr).
\end{align*}

\section{Bias corrected GSIs}\label{sec:biasc}

When we are interested in estimating
a linear combination of variance components,
then the corresponding GSI is a contrast.
Sometimes estimating a contrast requires
an additional function evaluation
per $(\bsx_i,\bsz_i)$ pair.
For instance the unbiased
estimator~\eqref{eq:mauntz} of $\ult^2_u$ requires three
function evaluations per pair compared to the two required
by the biased estimator of~\cite{jano:klei:lagn:node:prie:2012:tr}.

Proposition~\ref{prop:janonunb}
supplies a bias-corrected version of
\nocite{jano:klei:lagn:node:prie:2012:tr}
Janon et al.'s (2011) estimator of $\ult^2_u$
using only two function evaluations per $(\bsx_i,\bsz_i)$
pair.

\begin{proposition}\label{prop:janonunb}
Let $f\in L^2[0,1]^d$ and suppose that 
$\bsx_i,\bsz_i\simiid\dustd(0,1)^d$
for $i=1,\dots,n$ where $n\ge 2$.
Let $\bsy_i = \bsx_{i,u}\glu\bsz_{i,-u}$
and define 
\begin{align*}
\hat\mu & =\frac1n\sum_{i=1}^nf(\bsx_i),
&\hat\mu'& = \frac1n\sum_{i=1}^nf(\bsy_i),\\
s^2 & = \frac1{n-1}\sum_{i=1}^n(f(\bsx_i)-\hat\mu)^2,\quad\text{and}
&{s'}^{2} & = \frac1{n-1}\sum_{i=1}^n(f(\bsy_i)-\hat\mu')^2.
\end{align*}
Then
$\e\bigl( \wt\ult_u^2\bigr) = \ult_u^2$
where
$$
\wt\ult_u^2 = 
\frac{2n}{2n-1}
\biggl(
\frac1n\sum_{i=1}^nf(\bsx_i)f(\bsy_i)
- \Bigl(\frac{\hat\mu+\hat\mu'}2\Bigr)^2
+\frac{s^2+{s'}^2}{4n}
\biggr)
$$
\end{proposition}
\begin{proof}
First $\e( f(\bsx_i)f(\bsy_i) )=\mu^2+\ult^2_u$.
Next 
\begin{align*}
\e( (\hat\mu+\hat\mu')^2)
&= 4\mu^2+\var(\hat\mu)+\var(\hat\mu')+2\cov(\hat\mu,\hat\mu')\\
&= 4\mu^2+2\frac{\sigma^2}n+2\frac{\ult_u^2}{n}.
\end{align*}
Finally $\e(s^2) = \e({s'}^2)=\sigma^2$. Putting these together
yields the result.
\end{proof}

More generally, 
suppose that $\e\bigl(\tr(\Omega^\tran\wh\Theta)\bigr)$
contains a contribution of
$\mu^2\bsone^\tran\Omega\bsone$
which is nonzero if $\Omega$ is not a contrast.
Then a bias correction is available
for $\tr(\Omega^\tran\Theta)$.

\begin{proposition}\label{prop:unbiased}
Let $f\in L^2[0,1]^d$ and suppose that 
$\bsx_i,\bsz_i\simiid\dustd(0,1)^d$
for $i=1,\dots,n$ for $n\ge 2$.
For $u\subseteq\cd$ define 
\begin{align*}
\hat\mu_u & =\frac1n\sum_{i=1}^nf(\bsx_{i,u}\glu\bsz_{i,-u}),
\quad\text{and}\quad
{s^2_u}  = \frac1{n-1}\sum_{i=1}^n(f(\bsx_{i,u}\glu\bsz_{i,-u})-\hat\mu_u)^2.
\end{align*}
Then 
\begin{align}\label{eq:unbiased}
\frac{2n}{2n-1}
\sum_u\sum_v {\Omega_{uv}}
\Bigl(\wh\Theta_{uv}
- \Bigl(\frac{\hat\mu_u+\hat\mu_v}2\Bigr)^2
+\frac{s_u^2+{s_v}^2}{4n}
\Bigr)
\end{align}
is an unbiased estimate of
$\sum_u\sum_v\Omega_{uv}(\Theta_{uv}-\mu^2)$.
\end{proposition}
\begin{proof}
This follows by applying Proposition~\ref{prop:janonunb}
term by term.
\end{proof}

The computational burden for the
unbiased estimator in 
Proposition~\ref{prop:unbiased} is not
much greater than that for
the possibly biased estimator $\tr(\Omega^\tran\wh\Theta)$.  
It requires no additional function evaluations.
The quantities
$\hat\mu_u$ and $s_u^2$ need only be computed
for sets $u\subseteq\cd$ for which
$\Omega_{uv}$ or $\Omega_{vu}$ is nonzero
for some $v$.
If $\Omega$ is a sum of bilinear estimators
then the $-\sum_u\sum_v\Omega_{uv}\hat\mu_u\hat\mu_v/2$
cross terms also have that property.

The bias correction in estimator~\eqref{eq:unbiased}
complicates calculation of confidence
intervals for $\tr(\Omega^\tran\Theta)$.
Jackknife or bootstrap methods will work but
confidence intervals for contrasts are much
simpler because the estimators are simple averages.

\section{Comparisons}\label{sec:compare}

There is a $2^{2d}$--dimensional space of GSIs but only
a $2^d-1$--dimensional space of linear combinations
of variance components to estimate. As a result there
is more than one way to estimate a desired
linear combination of variance components.

As a case in point the Sobol' index $\ult^2_u$
can be estimated by either the original
method or by the contrast~\eqref{eq:mauntz}.
\cite{jano:klei:lagn:node:prie:2012:tr}
prove that their estimate of $\hat\mu$
improves on the simpler one and establish
asymptotic efficiency for their estimator
within a class of methods based on exchangeability,
but that class does not include the contrast.
Similarly, inspecting the Sobol' matrix
yields at least four ways to estimate
the variance component $\sigma^2_{\{1,2,3\}}$,
and superset importance can be estimated
via a square or a bilinear term.

Here we consider some theoretical aspects
of the comparison, but they do not lead
to unambiguous choices. Next we consider a
small set of empirical investigations.

\subsection{Minimum variance estimation}

Ideally we would like to choose
$\Omega$ to minimize the variance
of the sample GSI.
But, the variance of a GSI depends on fourth
moments of ANOVA contributions which
are ordinarily unknown and harder to
estimate than the variance components
themselves.

The same issue comes up in the estimation
of variance components, where MINQE
(minimum norm quadratic estimation)
estimators  were proposed in 
a series of papers by C.\ R.\ Rao
in the 1970s.
For a comprehensive treatment see~\cite{rao:klef:1988}
who present MINQUE and MINQIE versions
using unbiasedness or invariance as constraints.
The idea in MINQUE estimation is to minimize
a convenient quadratic norm as a proxy for
the variance of the estimator. 

The GSI context involves variance components
for crossed random effects models with interactions
of all orders.  Even the two way crossed
random effects model with an interaction is complicated
enough that no closed form estimator appears
to be known for that case.  See~\cite{klef:1980}. 

We can however generalize the idea behind
MINQE estimators to the GSI setting.
Writing
\begin{align*}
\var( \tr(\Omega^\tran\wh\Theta) )
& = 
\sum_{u\subseteq\cd}\sum_{v\subseteq\cd}
\sum_{u'\subseteq\cd}\sum_{v'\subseteq\cd}
\Omega_{uv}\Omega_{u'v'}
\cov\bigl(\wh\Theta_{uv},\wh\Theta_{u'v'}\bigr)
\end{align*}
we can obtain the upper bound
$$
\var( \tr(\Omega^\tran\wh\Theta) )
\le 
\biggl(\sum_{u\subseteq\cd}\sum_{v\subseteq\cd}|\Omega_{uv}|^2\biggr)
\max_{u,v}
\var\bigl(\wh\Theta_{uv}\bigr),
$$
leading to a proxy measure
\begin{align}\label{eq:soboproxy}
V(\Omega)=\sum_{u\subseteq\cd}\sum_{v\subseteq\cd}|\Omega_{uv}|^2
=\tr(\Omega^\tran\Omega).
\end{align}

Using the proxy for variance suggests
choosing the estimator which minimizes $C(\Omega)\times V(\Omega)$.
The contrast estimator~\eqref{eq:mauntz}
of $\ult^2_u$ has
$C(\Omega)\times V(\Omega) = 3\times 2=6$
while the original Sobol' estimator has
$C(\Omega)\times V(\Omega) = 2\times 1=2$.
The estimators~\eqref{eq:simplevarcomp}
and~\eqref{eq:bilin1} for $\sigma^2_{\{1,2,3\}}$
both have $V(\Omega)=8$. The former has cost
$C(\Omega)=9$, while the latter costs $C(\Omega)=6$.
As a result, the proxy arguments support
the original Sobol' estimator and
the alternative estimator~\eqref{eq:bilin1} for $\sigma^2_{\{1,2,3\}}$.


\subsection{Test cases}

To compare some estimators we use test functions
of product form:
\begin{align}\label{eq:prodtest}
f(\bsx) = \prod_{j=1}^d(\mu_j + \tau_jg_j(x_j))
\end{align}
where each $g_j$ satisfies
$$\int_0^1g(x)\rd x = 0,\quad
\int_0^1g(x)^2\rd x = 1,\quad\text{and}\quad
\int_0^1g(x)^4\rd x < \infty.
$$
The third condition ensures that all GSIs have finite
variance, while the first two allow us to write the
variance components of $f$ as
$$
\sigma^2_u = 
\begin{cases}
\prod_{j\in u}\tau^2_j\times\prod_{j\not\in u}\mu_j^2, & |u|>0\\
0, &\text{else,}
\end{cases}
$$
along with $\mu=\prod_{j=1}^d\mu_j$.

We will compare Monte Carlo estimates and so smoothness
or otherwise of $g_j(\cdot)$ play no role.  Only $\mu_j$,
$\tau_j$ and the third and fourth moments of $g$
play a role.   
Monte Carlo estimation is suitable
when $f$ is inexpensive to evaluate, like surrogate
functions in computer experiments. 
For our examples we take $g_j(x) = \sqrt{12}(x-1/2)$
for all $j$.

For an example function of non-product form, 
we take the minimum,
$$ f(\bsx) = \min_{1\le j\le d}x_j.$$
\cite{meandim} show that
$$\ult_u^2 = \frac{|u|}{(d+1)^2(2d-|u|+2)},$$
for this function.
Taking $u =\cd$, gives $\sigma^2 = d(d+1)^{-2}(d+2)^{-1}$.

\subsection{Estimation of $\sigma^2_{\{1,2,3\}}$}

We considered both simple and bilinear estimators
of $\sigma^2_{\{1,2,3\}}$ in Section~\ref{sec:specific}.
The simple estimator requires $9$ function evaluations
per $(\bsx,\bsz)$ pair, while three different bilinear ones
each require only $6$.

For a function of product form, all four of these
estimators yield the same answer for any specific
set of $(\bsx_i,\bsz_i)$ pairs. As a result the
bilinear formulas dominate the simple one
for product functions.

For the minimum function, with $d=5$ we find
that by symmetry,
$$\sigma^2_u = \ult^2_{\{1,2,3\}}
-3\ult^2_{\{1,2\}}+3\ult^2_{\{1\}}=\frac1{5940}\doteq 1.68\times10^{-4}.$$

Because we are interested in comparing the variance of
estimators of a variance, a larger sample is warranted
than if we were simply estimating a variance component.
Based on $1{,}000{,}000$ function evaluations we find
the estimated means and standard errors
are given in Table~\ref{tab:mintable}. We see that
the bilinear estimators give about half the standard
error of the simple estimator, corresponding to
about $(1.05/.571)^2\times 9/6\doteq 5.1$ times the statistical efficiency.

\begin{table}
\centering
\begin{tabular}{lcccc}
\toprule
Estimator & Simple & Bilin.$\{1\}$  & Bilin.$\{2\}$  & Bilin.$\{3\}$ \\
\midrule
Mean &$1.74\times 10^{-4}$ &$1.72\times 10^{-4}$ &$1.68\times 10^{-4}$ &$1.70\times 10^{-4}$\\
Standard error  &
$1.05\times 10^{-5}$ & $5.69\times10^{-6}$ & $5.71\times10^{-6}$ &$5.67\times10^{-6}$\\
\bottomrule
\end{tabular}
\caption{\label{tab:mintable}
Estimated mean and corresponding standard error for
three estimators of $\sigma^2_{\{1,2,3\}}$ for $f(\bsx)=\min_{1\le j\le 5}x_j$
when $\bsx\sim\dustd(0,1)^5$.
The Bilin.$\{1\}$ estimator is from equation~\eqref{eq:bilin1},
and the other Bilinear estimators are defined analogously.
}
\end{table}


\subsection{Estimation of $\ult^2_{\{1,2\}}$}

We consider two estimators of $\ult^2_u$.
The estimator~\eqref{eq:mauntz} is a bilinear contrast, averaging
$f(\bsx)(f(\bsx_u\glu\bsz_{-u})-f(\bsz))$.
The estimator~\eqref{eq:janons} using
the estimator of $\hat\mu$ from 
\cite{jano:klei:lagn:node:prie:2012:tr}
is a modification of 
Sobol's original simple estimator
based on averaging $f(\bsx)f(\bsx_u\glu\bsz_{-u})$.
The bias correction of Section~\ref{sec:biasc} makes an
asymptotically negligible difference, so
we do not consider it here.

The contrast estimator requires $3$ function evaluations
per $(\bsx,\bsz)$ pair, while Sobol's only requires $2$.
Both estimators make an adjustment to
compensate for the bias $\mu^2$.
Estimator~\eqref{eq:janons} subtracts an estimate
$\hat\mu^2$ based on combining all $2n$ function
evaluations, the square of the most natural
way to estimate $\mu$ from the available data.
Estimator~\eqref{eq:mauntz}
subtracts $(1/n^2)\sum_i\sum_{i'}f(\bsx_i)f(\bsx_{i,u}\glu\bsz_{i,-u})$,
which may be advantageous when
the difference $f(\bsx_u\glu\bsz_{-u})-f(\bsz)$
involves considerable cancellation,
as it might if $\bsx_u$ is unimportant. Thus we might
expect~\eqref{eq:mauntz} to be better when
$\ult^2_u$ is small.
We compare the estimators on a product
function, looking at $\ult^2_{u}$ for three
subsets $u$ of size $2$ and varying importance.


For the product function with $d=6$,
$\tau = (1,1,1/2,1/2,1/4,1/4)$, 
all $\mu_j=1$ for $j=1,\dots,6$,
and $g_j(x_j)=\sqrt{12}(x_j-1/2)$,
we may compute
$\ult^2_{\{1,2\}} = 3\doteq 0.50\sigma^2$,
$\ult^2_{\{3,4\}} \doteq 1.56\doteq 0.093\sigma^2$, and
$\ult^2_{\{5,6\}} \doteq 0.13\doteq 0.021\sigma^2$.

Results from $R=10{,}000$
trials with $n=10{,}000$ 
$(\bsx_i,\bsz_i)$ pairs each, are shown in
Table~\ref{tab:sobomaunprod}. 
The efficiency of the contrast estimator
compared to the simple one ranges
from about $0.5$ to about $2.5$
in this example, depending on the size
of the effect being estimated, with
the contrast being better for the small
quantity $\ult^2_{\{5,6\}}$.
\cite{sobo:tara:gate:kuch:maun:2007}
also report superiority of the contrast
estimator on a small $\ult^2_u$.

Neither estimator is always more efficient than the other,
hence no proxy based solely on $\Omega$
can reliably predict which of these is
better for a specific problem.

\begin{table}\centering
\begin{tabular}{lrrrrrr}
\toprule
$n=10{,}000$&\multicolumn{2}{c}{$\{1,2\}$}
&\multicolumn{2}{c}{$\{3,4\}$}
&\multicolumn{2}{c}{$\{5,6\}$}\\
\midrule
& Cont.&Simp. & Cont.&Simp. & Cont.&Simp. \\
\midrule
True & 3.0000 & 3.0000 & 0.5625 & 0.5625 & 0.1289 & 0.1289\\
Avg. & 3.0002 & 3.0002 & 0.5624 & 0.5628 & 0.1291 & 0.1294\\
Bias & 0.0002 & 0.0002 & $-$0.0005 & 0.0003 & 0.0002 & 0.0005\\
\midrule
S.Dev& 0.1325 & 0.1186 & 0.0800 & 0.0998 & 0.0378 & 0.0737
\\
Neg & $-$\phz & $-$\phz & $-$\phz & $-$\phz & 0.0001 & 0.0336\\
\midrule
Eff. &\multicolumn{2}{c}{$0.53$}&\multicolumn{2}{c}{$1.04$}&\multicolumn{2}{c}{$2.54$}\\
\bottomrule
 \end{tabular}
 \caption{\label{tab:sobomaunprod}
This table compares a simple estimator versus
a contrast for  $\ult^2_u$ with $n=10{,}000$.
The sets compared are $u=\{1,2\}$, $\{3,4\}$, and $\{5,6\}$
and $f$ is the product function described in the text.
The rows give the true values of $\ult^2_u$,
and for $10{,}000$ replicates,
the (rounded) sample values of their
average, bias, standard deviation and
proportion negative.  The last line is
estimated efficiency of the contrast,
$(2/3)$ times the ratio of the standard deviations squared.
 }
\end{table}


The bias correction from Section~\ref{sec:biasc}
makes little difference here because for $n=10{,}000$
there is very little bias to correct. It does make
a difference when $n=100$ (data not shown) but
at such small sample sizes the standard deviation
of $\wh\ult^2_u$ can be comparable to or larger
than $\ult^2_u$ itself for this function.

\subsection{Estimation of $\Upsilon^2_{\{1,2,3,4\}}$}

Here we compare two estimates of
$\Upsilon^2_{\{1,2,3,4\}}$, the square~\eqref{eq:upassquare}
and the bilinear estimator~\eqref{eq:bilshortcut2}
from Theorem~\ref{thm:bilinsuper}.
For a product function,
$\Upsilon^2_w = \prod_{j\in w}\tau^2_j\prod_{j\not\in w}(\mu_j^2+\tau^2_j).$
Squares have an advantage estimating small GSIs
so we consider one small and one large (for a four
way interaction) $\Upsilon^2$.

For $d=8$, $\tau = c(4,4,3,3,2,2,1,1)/4$
and all $\mu_j=1$ we find that
$\Upsilon^2_{\{1,2,3,4\}} \doteq 0.558 \doteq 0.0334\sigma^2$
and
$\Upsilon^2_{\{5,6,7,8\}} \doteq 0.00238 \doteq 0.000147\sigma^2$.
The bilinear estimate~\eqref{eq:bilshortcut2}
based on $w_1=\{1,2\}$ and $w_2=\{3,4\}$
for $\Upsilon_{\{1,2,3,4\}}$
(respectively  $w_1=\{5,6\}$ and $w_2=\{7,8\}$
for $\Upsilon_{\{5,6,7,8\}}$)
requires $C=7$ function evaluations, while
the square~\eqref{eq:upassquare} requires $C=16$.
From Table~\ref{tab:bilinvssqr} we see that
the square has an advantage
that more than compensates for using a larger
number of function evaluations and the
advantage is overwhelming for the smaller effect.

\begin{table}\centering
\begin{tabular}{lcc}
\toprule
& $\{1,2,3,4\}$& $\{5,6,7,8\}$\\
\midrule
Bilinear & $35.07$ & 4.019 \\
Square   & $\phz6.04$ & 0.051  \\
\midrule
Efficiency & $14.7\phz$ & $2{,}710$ \\
\bottomrule
\end{tabular}
\caption{\label{tab:bilinvssqr}
Standard errors for estimation of $\Upsilon^2_w$
by the bilinear estimate and a square
as described in the text.
The estimated
standard errors based on $n=1{,}000{,}000$
replicates are $10^{-3}$ times the values
shown. The relative efficiency of the square
is $7/16$ times the squared ratio of standard deviations.
}
\end{table}

The outlook for the bilinear estimator of $\Upsilon^2_w$ is
pessimistic. Its cost advantage grows with $|w|$;
for $|w|=20$ it has cost $1023$
compared to $2^{20}$ for the square. But $\Upsilon^2_w$
for such a large $w$ will often be so small that
the variance advantage from using a square will
be extreme.

\section{Conclusions}\label{sec:conc}

We have generalized Sobol' indices
to estimators of arbitrary linear
combinations of variance components.
Sometimes there are multiple ways
to estimate a generalized Sobol' index
with important efficiency differences.
Square GSIs where available
are very effective. 
When no square or sum of squares is available
a bilinear or low rank GSI can at least save some
function evaluations. Contrasts are simpler
to work than other GSIs, because they avoid bias corrections.


\section*{Acknowledgments}

This work was supported by the U.S.\ National
Science Foundation under grant DMS-0906056.
I thank Alexandra Chouldechova for translating
Sobol's description of the analysis of variance.
Thanks to Sergei Kucherenko for discussions on
Sobol' indices.
I also thank the researchers of the GDR MASCOT NUM
for an invitation to their 2012 meeting which
lead to the research presented here.

\bibliographystyle{apalike}
\bibliography{sensitivity}

\end{document}

%% file: newcommandsnotbook.tex
\allowdisplaybreaks

\newcommand{\cd}{{\cal D}}

\newcommand{\bsx}{\boldsymbol{x}}
\newcommand{\bsy}{\boldsymbol{y}}
\newcommand{\bsz}{\boldsymbol{z}}

\newcommand{\real}{\mathbb{R}}

\newcommand{\glu}{{\!\::\:\!}}

\newcommand{\tran}{\mathsf{T}} 

\newcommand{\phm}{\phantom{-}}
\newcommand{\phz}{\phantom{0}}

\newcommand{\mrd}{\mathrm{d}}
\newcommand{\rd}{\mathrm{\, d}}

\newcommand{\tr}{{\text{tr}}}

\newcommand{\var}{{\mathrm{Var}}}
\newcommand{\cov}{{\mathrm{Cov}}}

\newcommand{\wh}{\widehat}
\newcommand{\wt}{\widetilde}

\renewcommand{\emptyset}{\varnothing}
\renewcommand{\ge}{\geqslant}
\renewcommand{\le}{\leqslant}

\newcommand{\dustd}{\mathbf{U}} 

\newcommand{\e}{{\mathbb{E}}} 

